\documentclass[11pt, reqno]{amsart}

\usepackage{amssymb,latexsym,amsmath,amsfonts,amsthm}
\usepackage{mathrsfs}
\usepackage{enumitem}
\usepackage[usenames]{color}
\usepackage{hyperref}
\usepackage{comment}

\numberwithin{equation}{section}

\definecolor{DPurple}{rgb}{0.46,0.2,0.69}

\theoremstyle{definition}
\newtheorem{definition}{Definition}[section]
\newtheorem{question}[definition]{Question}

\theoremstyle{remark}
\newtheorem{remark}[definition]{Remark}

\theoremstyle{plain}
\newtheorem{theorem}[definition]{Theorem}
\newtheorem{result}[definition]{Result}
\newtheorem{lemma}[definition]{Lemma}

\newcommand{\eps}{\varepsilon}
\newcommand{\zt}{\zeta}
\newcommand{\zbar}{\overline{z}}

\newcommand{\cHess}{\mathfrak{H}_{\raisebox{-2pt}{$\scriptstyle {\mathbb{C}}$}}}
\newcommand{\dbar}{\overline\partial}
\newcommand{\dirder}{\boldsymbol{{\sf D}}}

\newcommand{\bdy}{\partial}
\newcommand{\OM}{\Omega}
\newcommand{\D}{\mathbb{D}}

\newcommand{\smoo}{\mathcal{C}}

\newcommand\leb[1]{\mathbb{L}^{{#1}}}

\newcommand{\bcdot}{\boldsymbol{\cdot}}
\newcommand{\lrarw}{\longrightarrow}

\newcommand{\gvar}{\mathcal{X}}
\newcommand{\germ}{\mathfrak{V}}
\newcommand{\multp}{\nu_{\raisebox{-3pt}{\!$\scriptstyle 0$}}}
\newcommand{\multexp}{\nu_{\raisebox{-2pt}{${\scriptscriptstyle 0}$}}}
\newcommand{\univ}{\boldsymbol{{\sf u}}}
\newcommand{\excep}{\mathscr{E}}

\newcommand{\Cn}{\mathbb{C}^n}

\newcommand{\C}{\mathbb{C}} 
\newcommand{\R}{\mathbb{R}}

\newcommand{\wt}{\widetilde}

\begin{document}

\title[Rarity of $\mathcal{C}^{1,1}$ solutions to the complex
Monge--Amp{\`e}re equation]{Rarity of $\boldsymbol{\mathcal{C}^{1,1}}$ solutions to the complex
Monge--Amp{\`e}re \\ equation on weakly pseudoconvex domains}

\author{Gautam Bharali}
\address{Department of Mathematics, Indian Institute of Science, Bangalore 560012, India}
\email{bharali@iisc.ac.in}

\author{Rumpa Masanta}
\address{Theoretical Statistics and Mathematics Unit, Indian Statistical Institute, Bangalore Centre,
Bangalore 560059, India}
\email{rumpa\_ra@isibang.ac.in}

\begin{abstract}
We show that on any weakly pseudoconvex $B$-regular domain, the classical Dirichlet problem for
the complex Monge--Amp{\`e}re equation with $\smoo^\infty$-smooth data does not in general admit 
$\smoo^{1,1}$-smooth solutions. This working draft is a prelude to potential-theoretic
solutions to some extension problems for mappings that were thought to rely on such
$\smoo^{1,1}$-smooth solutions.
\end{abstract}

\keywords{$B$-regular domain, complex Monge--Amp{\`e}re equation, reqular solutions}
\subjclass[2020]{Primary: 32T27, 32U05; Secondary: 32T40}

\maketitle

\vspace{-3.2mm}
\section{Introduction and statement of results}\label{S:intro}
A key part of this work is concerned with various aspects of the regularity of the solution
to the Dirichlet problem for the complex Monge--Amp{\`e}re equation. In its most classical form, the
problem is to solve the following: 
\begin{equation}\label{E:monge-ampere}
  \left.
  \begin{array}{r l}
  \underbrace{dd^cu\wedge\dots \wedge dd^cu}_\text{$n$ factors} =:
    (dd^c{u})^n &\mkern-9mu{= f\beta_n, \; 
    \text{ $u\in \smoo(\overline{\Omega})\cap {\sf psh}(\Omega)$},} \\
    u|_{\bdy\Omega} &\mkern-9mu{= \varphi,}
    \end{array} \right\}
\end{equation}
where $f\in \smoo(\overline\Omega)$ and is non-negative, and $\varphi\in \smoo(\bdy\Omega; \R)$.
Here, $d = (\bdy + \dbar)$ is the exterior derivative, $d^c := i(\bdy - \dbar)$, and
$\beta_n$ is defined as
\[
  \beta_n := (i/2)^n(dz_1\wedge d\zbar_1)\wedge\dots
  \wedge (dz_n\wedge d\zbar_n).
\]
We will not comment here on how, in general, the solution\,---\,assuming it exists\,---\,is
interpreted. We refer the reader to Section~\ref{S:ess} for a brief discussion on the technical
aspects of \eqref{E:monge-ampere}.
The significance of the left-hand side of
\eqref{E:monge-ampere}  
is that for any $\smoo^2$-smooth function $g$ on $\Omega$,
\begin{equation} \label{E:det}
  (dd^c{g})^n = A_n\det(\cHess{g})\beta_n, 
\end{equation}
where $A_n = 4^n n!$ and $\cHess{g}$ denotes the complex Hessian of $g$. For a part of this
working draft, we will examine the existence of better-than-$\smoo^1(\overline{\Omega})$ solutions
to the Dirichlet problem for the complex Monge--Amp{\`e}re problem.
\smallskip

Perhaps the most fundamental result on high-regularity solutions to the Dirichlet problem for the
complex Monge--Amp{\`e}re equation is that of Caffarelli--Kohn--Nirenberg--Spruck
\cite[Theorem~1.1]{caffarelliEtAl:DpnsoeeII85}, which states that if $\Omega$ is a bounded,
strongly pseudoconvex domain with $\smoo^\infty$-smooth boundary,
$f\in \smoo^\infty(\overline{\Omega})$, with $f> 0$, and $\varphi\in \smoo^\infty(\bdy\Omega; \R)$,
then \eqref{E:monge-ampere} (in fact, the above-cited result addresses a slightly more general
problem) has a unique strongly plurisubharmonic solution in
$\smoo^\infty(\overline{\Omega})$. One of reasons for the interest outside the realm of PDEs
in \cite[Theorem~1.1]{caffarelliEtAl:DpnsoeeII85} is its implications on the boundary
behaviour of proper holomorphic maps: see, for instance, \cite{krantz:espscdMAe00} and the
references therein. Such implications generated a lot of interest in the extension of
\cite[Theorem~1.1]{caffarelliEtAl:DpnsoeeII85} to weakly pseudoconvex domains in the 1990s, but
no categorical results.
\smallskip

Before further discussing weakly pseudoconvex domains, we introduce the notion of a \emph{$B$-regular
domain}. We refer the reader to Section~\ref{S:ess} for a definition of $B$-regularity and state
here its significance: on a $B$-regular domain $\Omega$, \eqref{E:monge-ampere} has
a unique solution in $\smoo(\overline{\Omega})\cap {\sf psh}(\Omega)$ for any $(\varphi, f)$ as specified
above. Our first few theorems will be on \emph{non-existence of high-regularity solutions}. These
will be stated for $B$-regular domains in order to rule out the non-existence of high-regularity 
solutions caused by the non-existence of $\smoo(\overline{\Omega})$ solutions. The extension of
\cite[Theorem~1.1]{caffarelliEtAl:DpnsoeeII85} to weakly pseudoconvex domains is a minor part of
this work and we discuss this only because it provides a context for a question that is one of
the foci of this working draft. We believe that it is known to experts that,
on a $\smoo^\infty$-smoothly bounded, weakly pseudoconvex $B$-regular domain, the unique solution to
\eqref{E:monge-ampere} is \textbf{not} in $\smoo^\infty(\overline{\Omega})$ for $\varphi\equiv 0$ and for
 $f\in \smoo^\infty(\overline{\Omega})$ with (as in
\cite[Theorem~1.1]{caffarelliEtAl:DpnsoeeII85}) $f>0$. Our first theorem vastly expands the
class of data $(\varphi, f)$ for which $\smoo^\infty(\overline{\Omega})$ solutions are forbidden and,
in fact, addresses $\smoo^2(\overline{\Omega})$ solutions.

\begin{theorem}\label{T:CMA_not-smooth}
Let $\Omega$ be a $B$-regular domain in $\Cn$, $n\geq 2$, with $\smoo^2$-smooth boundary
that is weakly pseudoconvex. Let $f\in \smoo^2(\overline{\Omega})$ and let $f> 0$.
Let $p\in \bdy\Omega$ and $\univ_p$ be a unit vector in
$H_p(\bdy\Omega)$ such that the Levi-form of $\bdy\Omega$ at $p$ is zero on
${\rm span}_{\C}\{\univ_p\}$. For any $\varphi: \bdy\Omega\lrarw \R$ given as
\[
  \varphi = -\Psi|_{\bdy\Omega},
\]
where $\Psi$ is plurisubharmonic in a neighbourhood of $\overline{\Omega}$, is
$\smoo^2$-smooth, and $\langle \univ_p, \cHess{\Psi}(p)^{\!{\sf T}}\univ_p\rangle > 0$,
the unique solution to the Dirichlet problem \eqref{E:monge-ampere} does not
belong to $\smoo^2(\overline{\Omega})$.
\end{theorem}

Here, $\langle\bcdot\,,\,\bcdot\rangle$ denotes the standard Hermitian inner product on $\Cn$.
\smallskip

One of the foci of this work is the following

\begin{question}\label{Q:C11_reg}
Let $\Omega$ be a $B$-regular domain in $\Cn$, $n\geq 2$, with $\smoo^\infty$-smooth boundary
that is weakly pseudoconvex. Given any $f\in \smoo^{1,1}(\overline{\Omega})$ with $f>0$ and
$\varphi\in \smoo^\infty(\bdy\Omega; \R)$, does the unique solution of the problem
\eqref{E:monge-ampere} belong to $\smoo^{1,1}(\overline{\Omega})$? How does this answer change
if $f\in \smoo^{\infty}(\overline{\Omega})$?
\end{question}

The above question is motivated by the potential applications of
\cite[Theorem~1.1]{caffarelliEtAl:DpnsoeeII85}. In view of Theorem~\ref{T:CMA_not-smooth}, it is
natural to seek solutions in $\smoo^{1,1}(\overline{\Omega})$. An affirmative answer to
Question~\ref{Q:C11_reg} would lead to results that are still interesting and for whose proofs 
no other approaches are currently known (but more on this later). Furthermore, if $\Omega$ and
$\varphi: \bdy\Omega\lrarw \R$ are exactly as in \cite[Theorem~1.1]{caffarelliEtAl:DpnsoeeII85}
but $f>0$ is merely of class $\smoo^{1,1}(\overline{\Omega})$, then 
\cite[Theorem~1.2]{caffarelliEtAl:DpnsoeeII85} shows that the unique solution to
\eqref{E:monge-ampere} belongs to $\smoo^{1,1}(\overline{\Omega})$. Its proof cannot be
replicated in weakly pseudoconvex settings. Thus, Question~\ref{Q:C11_reg} is also an
interesting question in its own right independent of its ramifications. This question is
settled as follows:

\begin{theorem}\label{T:CMA_not-C11}
Let $\Omega$ be a $B$-regular domain in $\Cn$, $n\geq 2$, with $\smoo^\infty$-smooth boundary that is
weakly pseudoconvex. Then, there exists a function $\varphi\in \smoo^\infty(\bdy\Omega; \R)$ such that,
for any
$f\in \smoo^{\infty}(\overline{\Omega})$ with $f>0$, the unique solution to the Dirichlet problem
\eqref{E:monge-ampere} is not in $\smoo^{1,1}(\overline\Omega)$.
\end{theorem}

For the data $(\varphi, f)$ stipulated in Theorem~\ref{T:CMA_not-smooth}, let ${\sf u}_{\varphi,f}$
denote the solution to the problem \eqref{E:monge-ampere}. If ${\sf u}_{\varphi,f}$ belonged to
$\smoo^{1,1}(\overline{\Omega})$ then, by Rademacher's Theorem, and the fact that
$\cHess{{\sf u}_{\varphi,f}}(z)$ is Hermitian wherever it is defined, we would infer that all
second-order derivatives are in $\leb{\infty}(\overline{\Omega})$. From this, it might seem, given the
conclusion of Theorem~\ref{T:CMA_not-C11}, that one can rule out ${\sf u}_{\varphi,f} \in 
\smoo^{1,1}(\overline{\Omega})$ by purely classical considerations. We shall address this notion in
Remark~\ref{rem:C11_Rademacher}. The proof of Theorem~\ref{T:CMA_not-C11} requires a new idea. In fact,
the crux of the proof of Theorem~\ref{T:CMA_not-C11} is an argument that provides a more refined
conclusion.

\begin{theorem}\label{T:CMA_rel-to-type}
Let $\Omega$ be a $B$-regular domain in $\Cn$, $n\geq 2$, with $\smoo^\infty$-smooth boundary.
Suppose there exists a point $\xi\in \bdy\Omega$ whose D'Angelo $1$-type
$\tau_1(\xi) = 2k$, $k\geq 2$. Then,
  there exists a function $\varphi\in \smoo^\infty(\bdy\Omega; \R)$ such that, for
  any $f\in \smoo^\infty(\overline{\Omega})$ with $f>0$, the unique solution to the Dirichlet
  problem \eqref{E:monge-ampere} is not in $\smoo^{0,\alpha}(\overline\Omega)$
  for any $\alpha\in (1/k,1]$.  
\end{theorem}
 
\begin{remark}
Under the hypothesis of Theorem~\ref{T:CMA_rel-to-type}, $\Omega$ is pseudoconvex; see, for instance,
\cite[Section~2]{sibony:cdp87}. It then follows from \cite[Section~5]{dangelo:rhoca82} that for a point
in $\bdy\Omega$ of finite D'Angelo $1$-type, this type is an even number. This is tacit in our
assumption on $\xi$ above.
\end{remark}

The negative conclusions of Theorems~\ref{T:CMA_not-smooth}, \ref{T:CMA_not-C11},
and~\ref{T:CMA_rel-to-type} prompt the following

\begin{question}\label{Q:appl}
Are any of the mapping problems alluded to above tractable in the absence of results assuring high-regularity
solutions to the problem \eqref{E:monge-ampere} with smooth data?
\end{question}

We shall add a couple of applications to this working draft that would answer Question~\ref{Q:appl}.

\section{Geometric preliminaries}\label{S:geom}
We begin with a list of notations for basic geometric objects that will appear frequently
in the sections that follow (some of which were used without comment in Section~\ref{S:intro}).
\begin{itemize}[leftmargin=25pt]
  \item[$(1)$] For $v\in \Cn$, $\|v\|$ will denote the Euclidean norm of $v$.
  \vspace{0.45mm}
  
  \item[$(2)$] Given a non-empty set $S\subseteq \Cn$ and $z\in \Cn$,
  \[
    {\rm dist}(z, S) := \inf\{\|z-x\|: x\in S\}.
  \]

  \item[$(3)$] Given a point $z \in \Cn$ and $r>0$, $\mathbb{B}^n(z,r)$ will denote the open Euclidean
  ball in $\Cn$ with radius $r$ and centre $z$. For simplicity, we will write 
  $\mathbb{B}^n := \mathbb{B}^n(0,1)$
  and $\D := \mathbb{B}^1(0,1)$.
  \vspace{0.45mm}

  \item[$(4)$] With $z$ and $r$ as above, $\mathbb{B}^n(z,r)^* := \mathbb{B}^n(z,r)\setminus \{z\}$. 
\end{itemize}
\smallskip

The remainder of this section is dedicated to a brief introduction to the D'Angelo $1$-type. In doing
so, we shall follow the discussion in \cite{dangelo:rhoca82}. Consider a $\smoo^\infty$-smooth
real hypersurface
$M\subset \Cn$, $n\geq 2$, and let $\xi\in M$. A \emph{germ of an analytic variety at $\xi$} refers to
a closed analytic subvariety of some neighbourhood $\mathscr{N}_{\xi}$ of $\xi$ (the word ``germ''
conveying that the neighbourhood $\mathscr{N}_{\xi}$ is not relevant to the discussion).
\smallskip

With $M$ and $\xi$ as above, consider a germ $\gvar$ of a $1$-dimensional variety at $M$. By the
Puiseux parametrisation, there exists a planar neighbourhood $U$ of $0\in \C$ and a holomorphic
map $\psi = (\psi_1,\dots ,\psi_n): (U,0)\lrarw (\Cn,\xi)$ such that $\gvar = {\sf image}(\psi)$.
Write
\[
  \multp(\psi) := \min\{{\rm mult}_0(\psi_1-\xi_1),\dots, {\rm mult}_0(\psi_n-\xi_n)\},
\]
where we write $\xi = (\xi_1,\dots, \xi_n)$ and ${\rm mult}_0(\psi_j-\xi_j)$ denotes the multiplicity
of the zero at $\zt=0$ of the function $(\psi_j-\xi_j)$, $j=1,\dots, n$. The quantity
$\multp(\rho\circ\psi)$ has an analogous meaning, where $\rho$ is a defining function for $M$ in a
neighbourhood of $\xi$. Namely:
\[
  \multp(\rho\circ\psi) := \inf\left\{\alpha+\beta:
  \partial_{\zt}^{\alpha}\partial^{\beta}_{\overline{\zt}}(\rho\circ\psi)(0)\neq 0\right\}.
\]
Thus, if all the Taylor coefficients at $\zt=0$ of $\rho\circ\psi$ vanish, then
$\multp(\rho\circ\psi) = \infty$.
Finally, the \emph{order of contact of $\gvar$ with $M$
at $\xi$} is defined as
\[
  \multp(\rho\circ\psi)/\multp(\psi).
\]
One can check\,---\,see \cite[Chapter~2]{dangelo:scvgrh93}, for instance\,---\,that the above ratio
does not depend on the choices of $(\psi, \rho)$.
\smallskip

Here, and in later sections, $\germ_{\xi}$ will denote the collection of all $1$-dimensional varieties
at $\xi$. We can now present the following definition.

\begin{definition}
Let $M$ be a $\smoo^\infty$-smooth real hypersurface in $\Cn$, $n\geq 2$, and let $\xi\in M$.
The \emph{D'Angelo $1$-type of $M$ at $\xi$}, denoted by $\tau_1(M,\xi)$ (and written simply as
$\tau_1(\xi)$ if the hypersurface in question is unambiguous), is defined as
\[
  \tau_1(\xi) := \sup\nolimits_{\gvar\in \germ_{\xi}}\frac{\multp(\rho\circ\psi_{\gvar})}{\multp(\psi_{\gvar})},
\]
where $\psi_{\gvar}: (U_{\gvar},0)\lrarw (\Cn,0)$ is a holomorphic parametrisation of $\gvar$ (mapping
$0$ to $\xi$).   
\end{definition}

\section{Essential analytical results}\label{S:ess}
We begin with a discussion on $B$-regular domains and its significance that was hinted at
in Section~\ref{S:intro}.

\begin{definition}\label{D:b-regular}
Let $\Omega\varsubsetneq\Cn$ be a bounded domain. We say that $\Omega$ is \emph{$B$-regular} if $\bdy\Omega$
is a \emph{$B$-regular set}: i.e., if each function $\varphi\in\smoo(\bdy\Omega;\R)$ is the uniform
limit on $\bdy\Omega$ of a sequence
$(u_{\nu})_{\nu\geq 1}$ of continuous plurisubharmonic functions defined
on open neighbourhoods of $\bdy\Omega$ (each such neighbourhood 
depending on the function $u_{\nu}$).
\end{definition}

The above definition is taken from \cite{sibony:cdp87}. The definition of a $B$-regular domain
in some later papers seems different from that in
Definition~\ref{D:b-regular}; however, \cite[Theorem~2.1]{sibony:cdp87} establishes 
the equivalence between them (see also \cite[Theorem~1.7]{blocki:cmaohd96}).
\smallskip

As mentioned in Section~\ref{S:intro}, the Dirichlet problem~\eqref{E:monge-ampere} admits a unique
solution for any non-negative function $f\in\smoo(\overline\Omega)$ and any boundary data
$\varphi\in\smoo(\bdy\Omega;\R)$. This was established by B{\l}ocki; see \cite[Theorem~4.1]{blocki:cmaohd96}. When, for the solution $u$, $u|_\Omega$ is not of class $\smoo^2(\Omega)$,
the left-hand side of
\eqref{E:monge-ampere} must be interpreted as a current of bidegree $(n,n)$. This interpretation is well defined
for $u\in\smoo(\Omega)\cap{\sf psh}(\Omega)$, as established by Bedford--Taylor
\cite{bedfordtaylor:dpcmae76}, who proved an
existence and uniqueness theorem for the Dirichlet problem~\eqref{E:monge-ampere} with the above-mentioned
data for strongly
pseudoconvex domains.
As for \cite[Theorem~4.1]{blocki:cmaohd96}: it is a consequence of Theorem~8.3
in \cite{bedfordtaylor:dpcmae76}.
\smallskip

Next, we consider a notion that, in the specific form below, was introduced\,---\,to the best of our
knowledge\,---\,in \cite{fornaesssibony:cpshwpd89}. In \cite[Section~3]{fornaesssibony:cpshwpd89}, this notion,
which the authors call the ``type,'' is examined for domains with smooth boundary; in this setting, equivalent
expressions have appeared in the literature in the early 1980s. Let $\Omega$ be a domain in $\Cn$,
$n\geq 2$, and let $\xi\in \bdy\Omega$. For any germ $\gvar$ of a $1$-dimensional analytic variety at
$\xi$, define
\[
  \tau(\xi,\gvar) := \sup\left\{s>0 : \limsup_{\gvar\setminus \{\xi\}\,\ni\,z\to \xi}
                            \frac{{\rm dist}(z, \bdy\Omega)}{\|z-\xi\|^s} < \infty\right\}.
\]
This notion is relevant to the proof of Theorem~\ref{T:CMA_rel-to-type}. The key result that we shall
need is the following. In what follows, we shall use the notation introduced in Section~\ref{S:geom}.

\begin{result}[Forn{\ae}ss--Sibony, {\cite[Proposition~3.1]{fornaesssibony:cpshwpd89}}]\label{R:type_ineq}
Let $\Omega\varsubsetneq \Cn$, $n\geq 2$, be a bounded domain with $\smoo^\infty$-smooth boundary and let 
$\xi\in \bdy\Omega$. Assume there exists a function $\phi\in \smoo(\overline{\Omega})\cap {\sf psh}(\Omega)$
such that, for constants $\beta\in (0,1]$ and $C>0$,
\begin{itemize}
  \item[$(a)$] $|\phi(z')-\phi(z)|\leq C\|z-z'\|^\beta$ for all $z, z'\in \Omega$,
  \item[$(b)$] $\phi(0)=0$ and $\phi(z)\leq -\|z-\xi\|^{2k\beta}$ for all $z\in \overline{\Omega}$.
\end{itemize}
Then, $\sup_{\gvar\in \germ_{\xi}}\tau(\xi, \gvar)\leq 2k$.
\end{result}

Observe that $\tau(\xi,\gvar)$ makes sense regardless of whether or not $\bdy\Omega$ is smooth. It
turns out that when $\bdy\Omega$ is smooth,
$\sup_{\gvar\in \germ_{\xi}}\tau(\xi, \gvar)$ equals the D'Angelo $1$-type of $\bdy\Omega$ at $\xi$. We could
not find a proof of this fact in the literature. As it is essential to this work, we provide a proof
here.

\begin{lemma}\label{L:rough_equals_1-type}
Let $\Omega\varsubsetneq \Cn$, $n\geq 2$, be a domain with $\smoo^\infty$-smooth boundary and let 
$\xi\in \bdy\Omega$. 
For any germ $\gvar$ of a $1$-dimensional analytic variety at $\xi$,
\[
  \tau(\xi, \gvar) = \text{the order of contact of $\gvar$ with $\bdy\Omega$ at $\xi$}.
\]  
Therefore, $\sup_{\gvar\in \germ_{\xi}}\tau(\xi, \gvar) = \tau_1(\xi)$.
\end{lemma}
\begin{proof}
Fix $\xi\in \bdy\Omega$ and $\gvar\in \germ_{\xi}$. 
We may assume $\xi=0$.
There exists $r_1>0$ such that $\gvar$ is a closed analytic subvariety of $\mathbb{B}^n(0,r_1)$.
It is a classical fact that, shrinking $r_1>0$ if needed, there exist a planar neighbourhood $U$ of $0\in \C$
with ${\rm diam}(U)<1$ and
$\psi = (\psi_1,\dots, \psi_n)\in {\rm Hol}(U; \Cn)$ such that $\psi(0)=0$ and $\gvar = {\sf image}(\psi)$.
Fix a defining function
$\rho$ for $\Omega$. Then:
\[
  \text{the order of contact of $\gvar$ with $\bdy\Omega$ at $0$} :=
  \frac{\multp(\rho\circ \psi)}{\multp(\psi)}.  
\]
Recall, from the discussion in Section~\ref{S:geom},
that the right-hand side (call it
$t$) above does not depend on the choices of $(\psi, \rho)$. 
Since $\gvar$ is $1$-dimensional and since the zeros of \emph{at least} one $\psi_j$, $1\leq j\leq n$, are
isolated points in $U$, there exist 
constants $B>1$ and $r_2>0$ such that
\begin{equation}\label{E:compare}
  B^{-1}|\zt|^{\multexp(\psi)}\leq \|\psi(\zt)\|\leq  B|\zt|^{\multexp(\psi)}
  \quad \forall \zt: |\zt| < r_2. 
\end{equation}

Let $\widehat{\Omega}:=\bigcup_{z\in \overline{\Omega}}(z+\mathbb{B}^n)$.
There exists a constant $\kappa>1$ such that
\begin{equation}\label{E:dist}
  \kappa^{-1}|\rho(z)|\leq {\rm dist}(z, \bdy\OM)\leq \kappa|\rho(z)| \quad \forall z\in \widehat{\Omega}. 
\end{equation}
Let $\sigma>0$ be such that $0 < t-2\sigma < t$. There exists a constant $C>0$ such that
for all $\eps\in (0,r_1)$,
\begin{align}
  |\rho\circ \psi(\zt)| &< C|\zt|^{\multexp(\rho\circ\psi)} \notag \\
  &= C|\zt|^{t\multexp(\psi)} < C|\zt|^{(t-\sigma)\multexp(\psi)}
  \quad \forall \zt\in \psi^{-1}(\mathbb{B}^n(0,\eps)^*). \label{E:small}
\end{align}
The inequality in $\eqref{E:small}$ is due to the fact that ${\rm diam}(U)<1$.
Combining the last estimate with \eqref{E:compare} and \eqref{E:dist}, we get
\[
  \frac{{\rm dist}(\psi(\zt), \bdy\Omega)}{\|\psi(\zt)\|^{t-\sigma}} < \kappa B^{t-\sigma}C
  \quad \forall \zt \in \psi^{-1}(\mathbb{B}^n(0,\eps)^*),
\]
which holds for any $\eps>0$ sufficiently small. 
This implies that
\[
  t-\sigma \in \left\{s>0 : \limsup_{\gvar\setminus \{0\}\,\ni\,z\to 0}
                            \frac{{\rm dist}(z, \bdy\Omega)}{\|z\|^s} < \infty\right\} =: \mathcal{S}(\gvar).
\]

From the last assertion, if follows that for each $\sigma>0$ considered in the last paragraph,
$t-2\sigma$ is not an upper bound of $\mathcal{S}(\gvar)$. We now argue that $t$ is an upper bound of
$\mathcal{S}(\gvar)$. If $t = +\infty$, then we are done. Therefore, assume that $t$ is finite.
Assume, if possible, that $t$ is not an upper bound of $\mathcal{S}(\gvar)$. Then, there exists an
$s\in \mathcal{S}(\gvar)$ such that $s>t$. As $t$ is finite, $\multp(\rho\circ \psi)$ is finite.
So, for each $\eps\in (0,r_1)$, by definition, 
$(\gvar\setminus \bdy\Omega)\cap \mathbb{B}^n(0,\eps) \neq \emptyset$.
By the definition of $\multp(\rho\circ\psi)$,
there exists a constant $c>0$ and, for each $\eps\in (0,r_1)$, a point $\zt_{\eps}\in 
\psi^{-1}(\mathbb{B}^n(0,\eps)^*)$ such that
\[
  |\rho\circ \psi(\zt_{\eps})| \geq c|\zt_{\eps}|^{\multexp(\rho\circ \psi)}.
\]
Combining this with \eqref{E:compare} and \eqref{E:dist}, we get
\[
  \frac{{\rm dist}(\psi(\zt_{\eps}), \bdy\Omega)}{\|\psi(\zt_{\eps})\|^s} 
  \geq \kappa^{-1}B^{-s}\frac{|\rho\circ \psi(\zt_{\eps})|}{|\zt_{\eps}|^{s\multexp(\psi)}}
  > c\kappa^{-1}B^{-s}|\zt_{\eps}|^{-(s-t)\multexp(\psi)}
\]
for all $\eps>0$ sufficiently small. But since $|\zt_{\eps}|\searrow 0$ as $\eps\searrow 0$,
the above contradicts the fact that $s\in \mathcal{S}(\gvar)$. Hence, $t$ is an upper bound
of $\mathcal{S}(\gvar)$. It follows from the assertions in this paragraph that
$t = \sup\mathcal{S}(\gvar)$. Now, by the definition of the D'Angelo $1$-type,
$\sup_{\gvar\in \germ_{\xi}}\tau(\xi, \gvar) = \tau_1(\xi)$.
\end{proof}

\section{The proof of Theorem~\ref{T:CMA_not-smooth}}
Let $V\supset\overline\Omega$ be a domain
such that $\Psi\in\smoo^2(V)$ and $\Psi$ is plurisubharmonic on $V$.
If possible, let the unique solution to the Dirichlet problem \eqref{E:monge-ampere}, with
the data $(\varphi,f)$ as prescribed, belong to
$\smoo^2(\overline{\Omega})$. Then  
there exists an open set $V^*\supset\overline\Omega$ and an extension $\wt u\in \smoo^2(V^*)$
of this solution. Since $f>0$ on $\overline{\Omega}$, it follows that
\begin{equation}\label{E:wtu-strict}
  \langle v, \cHess{\wt u}(z)^{\!{\sf T}}v\rangle \geq 0\quad\forall z\in \overline{\Omega}, \; \ \forall v\in\C^n.
\end{equation}
Let $W:=V\cap V^*$.
Define the function $\wt\Psi:W\lrarw\R$ by $\wt\Psi:=\wt u+\Psi$. Clearly, $\wt\Psi\in\smoo^2(W)\cap {\sf psh}(\Omega)$ and 
$\wt\Psi\equiv0$ on $\bdy\Omega$.  
As $\langle \univ_p, \cHess{\Psi}(p)^{\!{\sf T}}\univ_p\rangle > 0$, in view of \eqref{E:wtu-strict} we have
\begin{equation}\label{E:levi-form-strict}
  \langle \univ_p, \cHess{\wt\Psi}(p)^{\!{\sf T}}\univ_p\rangle > 0.
\end{equation}
So, $\wt\Psi$ is non-constant on $\overline\Omega$. Therefore, as $\wt\Psi\in\sf{psh}(\Omega)$ and
$\wt\Psi|_{\bdy\Omega}=0$, 
by the maximum principle it follows that $\wt\Psi<0$ on $\Omega$.
Let $\xi\in\bdy\Omega$. Since $\wt\Psi(\xi)=0=\sup_{\Omega}\wt\Psi$, by the classical Hopf Lemma we have
$\partial\wt\Psi(\xi)/\partial\eta>0$, where $\eta$ is the outward normal vector field of $\bdy\Omega$. 
Recall that
\[
  \partial\wt\Psi(\xi)/\partial\eta = \big(\,\nabla\wt\Psi(\xi)\!\mid\! \eta(\xi)\,\big)
\]
for every $\xi\in \bdy\Omega$, where $(\bcdot\,|\,\bcdot)$ denotes the standard real inner product
on $\R^{2n}$. Therefore,
$\nabla\wt\Psi(\xi)\neq 0$ for every $\xi\in\bdy\Omega$. Now, as $\wt\Psi < 0$ on $\Omega$,
$\wt\Psi$ is a defining function for $\Omega$.
Therefore, by \eqref{E:levi-form-strict}, the Levi-form of $\bdy\Omega$ at $p$ on ${\rm span}_{\C}\{\univ_p\}\setminus \{0\}$
is strictly positive, which contradicts our hypothesis.
Hence, the result follows. \hfill $\qed$
\medskip

\section{The proofs of Theorems~\ref{T:CMA_not-C11} and~\ref{T:CMA_rel-to-type}}
Before we give the proof of Theorem~\ref{T:CMA_not-C11}, let us elaborate upon the observation in
Section~\ref{S:intro} that, with $(\varphi, f)$ as in Theorem~\ref{T:CMA_not-smooth}, if
${\sf u}_{\varphi,f}$ denotes the unique solution of the Dirichlet problem \eqref{E:monge-ampere}, then
the conclusion ${\sf u}_{\varphi,f} \in \smoo^{1,1}(\overline{\Omega})$ can be ruled out in view of
Theorem~\ref{E:monge-ampere} and Rademacher's Theorem.

\begin{remark}\label{rem:C11_Rademacher}
Let $\Omega$ be as in Theorem~\ref{T:CMA_not-smooth}
and \textbf{fix} data $(\varphi, f)$ as in that
theorem. Let ${\sf u}_{\varphi,f}$ be as defined above and assume that
${\sf u}_{\varphi,f} \in \smoo^{1,1}(\overline{\Omega})$. Then, by Rademacher's Theorem, there
exists a set $\excep\varsubsetneq \overline{\Omega}$ such that $|\excep| = 0$ (here, $|\bcdot|$
denotes the Lebesgue measure on $\R^{2n}$) and such that ${\sf u}_{\varphi,f}$ is twice-differentiable
on $\overline{\Omega}\setminus\excep$. By Theorem~\ref{T:CMA_not-smooth}, $\excep\neq \emptyset$.
Rademacher's Theorem gives us no information on the structure of the set $\excep$. Thus, defining
\[
  w(\Omega) := \{\xi\in \bdy\Omega : \ \text{$\bdy\Omega$ is weakly Levi-pseudoconvex at $\xi$}\},
\]
there is no reason to assume (even if we apply a version of Rademacher's Theorem to the manifold
$\bdy\Omega$) that $w(\Omega)\setminus\excep\neq \emptyset$. If $w(\Omega)\subset \excep$, we
leave it to the reader to check that the part of our proof of Theorem~\ref{T:CMA_not-smooth}
leading to a contradiction is uninformative if we begin with the assumption
${\sf u}_{\varphi,f}\in \smoo^{1,1}(\overline{\Omega})$. Now, \emph{still assuming}
$w(\Omega)\subset \excep$, we could attempt a more ``basic'' argument. A 
\emph{holomorphic directional derivative} will refer to the differential operator in
$T^{1,0}_{z}\Cn$, for some $z\in \Cn$, canonically associated with a vector $v\in \Cn$ with
$\|v\|=1$. For each $z\in \Omega\setminus \excep$, fix a direction $v(z)$ in the eigenspace
associated with the least eigenvalue of $\cHess{\wt\Psi}(z)$, and set
\[
 \dirder_z := \text{the holomorphic directional derivative at $z$ associated with $v(z)$.}
\]
Since $f>0$, $\det(\cHess{{\sf u}_{\varphi,f}})\in \leb{\infty}(\overline{\Omega})$, and
$\cHess{{\sf u}_{\varphi,f}}(z)$ is Hermitian at each $z\in \overline{\Omega}\setminus\excep$, 
it seems one would be led to
the desired contradiction if one could answer in the affirmative the question (with $p\in \bdy\Omega$
as in Theorem~\ref{T:CMA_not-smooth}):
\emph{Does there exist a sequence $\{z_{\nu}\}\subset \Omega\setminus \excep$ converging to $p$ such
that $(a)$~${\rm span}_{\C}\{\,v(z)\,\}\lrarw {\rm span}_{\C}\{\,\univ_p\,\}$ in $\mathbb{P}^{n-1}$,
and $(b)$~$\lim_{\nu\to \infty}|\overline{\dirder}_{z_{\nu}}\dirder_{z_{\nu}} {\sf u}_{\varphi,f}|=0$?}
Rademacher's Theorem is not sufficiently informative to answer $(a)$. As for $(b)$: in general
$\overline{\dirder}_{z}\dirder_{z}{\sf u}_{\varphi,f}$ is merely \emph{measurable}, whence it is
unclear why the stated limit must exist (let alone equal $0$).
\end{remark}

As alluded to in Section~\ref{S:intro}, the crux of the proof of Theorem~\ref{T:CMA_not-C11} is the argument
for Theorem~\ref{T:CMA_rel-to-type}. Thus, we shall begin with

\begin{proof}[The proof of Theorem~\ref{T:CMA_rel-to-type}]
\textbf{Fix} an $f\in \smoo^\infty(\overline\Omega)$ with $f>0$. Write $\Psi(z) := \|z\|^2$, $z\in \Cn$.
Take $\varphi: \bdy\Omega\lrarw \R$ as $\varphi := -\Psi|_{\bdy\Omega}$; clearly,
$\varphi\in \smoo^\infty(\bdy\Omega; \R)$. Let ${\sf u}_{\varphi,f}$ denote the unique solution 
in $\smoo(\overline{\Omega})$ of the Dirichlet problem \eqref{E:monge-ampere} with the above $(\varphi,f)$. If
\[
  {\sf u}_{\varphi,f}\notin \bigcup_{\alpha\in (0,1]}\smoo^{0,\alpha}(\bdy\Omega; \R),
\]
then we are done. Therefore, suppose ${\sf u}_{\varphi,f}\in \smoo^{0,\alpha}(\bdy\Omega; \R)$ for some
$\alpha\in (0,1]$. Now, define $\wt{\Psi} := {\sf u}_{\varphi,f}+\Psi|_{\overline{\Omega}}$.
As $\wt{\Psi}\in \smoo(\overline{\Omega})\cap {\sf psh}(\Omega)$ and $\wt{\Psi}|_{\bdy\Omega}\equiv 0$,
$\wt{\Psi}\leq 0$ by the maximum principle.
\smallskip

Define $\phi(z) := 2\wt{\Psi}(z)-\|z-\xi\|^2$ for $z\in \overline{\Omega}$.
A computation of $dd^{c}\phi$ in the sense of currents reveals that $\phi\in 
\smoo(\overline{\Omega})\cap {\sf psh}(\Omega)$. By construction, $\phi(\xi)=0$. As $\wt{\Psi}\leq 0$, we
have
\begin{equation}\label{E:upper_est}
  \phi(z) \leq -\|z-\xi\|^2 \quad \forall z\in \overline{\Omega}.
\end{equation}
As ${\sf u}_{\varphi,f}\in \smoo^{0,\alpha}(\overline{\Omega})$, so is $\phi$. Thus,
there exists a constant $C>0$ such that
\begin{equation}\label{E:diff_est}
  |\phi(z')-\phi(z)| \leq C\|z'-z\|^\alpha \quad \forall z, z'\in \overline{\Omega}.
\end{equation}
Assuming that $\alpha > 1/k$, it would follow from Result~\ref{R:type_ineq} that
$\sup_{\gvar\in \germ_{\xi}}\tau(\xi, \gvar)\leq 2/\alpha < 2k$. But this, due to
Lemma~\ref{L:rough_equals_1-type}, would imply that
\[
  \tau_1(\xi) = \sup_{\gvar\in \germ_{\xi}}\tau(\xi, \gvar) < 2k,
\]
which contradicts our hypothesis. Thus, $\alpha\notin (1/k, 1]$; hence the result. 
\end{proof}

We now give the

\begin{proof}[The proof of Theorem~\ref{T:CMA_not-C11}]
\textbf{Fix} an $f\in \smoo^\infty(\overline\Omega)$ with $f>0$. Let $\Psi$, $\varphi$, and
${\sf u}_{\varphi,f}$ be \textbf{exactly} as in the above proof. If possible, let
${\sf u}_{\varphi,f} \in \smoo^{1,1}(\overline{\Omega})$. Then, clearly,
${\sf u}_{\varphi,f}\in \smoo^{0,1}(\overline{\Omega})$.
\smallskip

By hypothesis, $\Omega$ is weakly pseudoconvex; pick a point $\xi\in \bdy\Omega$ at which
$\bdy\Omega$ is weakly Levi-pseudoconvex. Define $\phi(z) := 2\wt{\Psi}(z)-\|z-\xi\|^2$ for $z\in
\overline{\Omega}$. Since (by  assumption)
${\sf u}_{\varphi,f}\in \smoo^{0,1}(\overline{\Omega})$, by exactly the same argument as in the above
proof, we deduce that
\begin{align*}
  |\phi(z')-\phi(z)| &\leq C\|z'-z\| \quad \forall z, z'\in \overline{\Omega}, \\
  \phi(z) &\leq -\|z-\xi\|^2 \quad \forall z\in \overline{\Omega},
\end{align*}
for some $C>0$. Then, by Result~\ref{R:type_ineq} and Lemma~\ref{L:rough_equals_1-type}, we
conclude that the D'Angelo $1$-type $\tau_1(\xi) = 2$, which produces a contradiction. Thus
our assumption above is false, and the result follows.
\end{proof}
\smallskip

\section*{Acknowledgements}
G. Bharali is supported by a DST-FIST grant (grant no.~DST FIST-2021
[TPN-700661]). R. Masanta is supported by the Theoretical Statistics and Mathematics Unit at the Indian
Statistical Institute, Bangalore Centre.

\end{document}